\documentclass[12pt]{amsart}
\usepackage{latexsym}
\usepackage{amssymb}
\usepackage[colorlinks]{hyperref}

\newtheorem{theorem}{Theorem}[section]
\newtheorem{lemma}[theorem]{Lemma}
\newtheorem{proposition}[theorem]{Proposition}
\newtheorem{corollary}[theorem]{Corollary}
\newtheorem{conjecture}[theorem]{Conjecture}

\theoremstyle{definition}

\theoremstyle{remark}

\numberwithin{equation}{theorem}

\newcommand{\PGL}{{\mathrm {PGL}}}
\newcommand{\SL}{{\mathrm {SL}}}
\newcommand{\PSL}{{\mathrm {PSL}}}

\newcommand{\PGU}{{\mathrm {PGU}}}

\newcommand{\PSU}{{\mathrm {PSU}}}

\newcommand{\Sp}{{\mathrm {Sp}}}
\newcommand{\PSp}{{\mathrm {PSp}}}

\newcommand{\Ker}{\operatorname{Ker}}
\newcommand{\Aut}{{\mathrm {Aut}}}

\newcommand{\Irr}{{\mathrm {Irr}}}

\newcommand{\RR}{{\mathbb R}}

\newcommand{\FF}{{\mathbb F}}

\newcommand{\Center}{\mathbf{Z}}

\newcommand{\acd}{\mathrm{acd}}

\newcommand{\bC}{{\mathbf{C}}}
\newcommand{\bO}{{\mathbf{O}}}
\newcommand{\bN}{{\mathbf{N}}}
\newcommand{\bZ}{{\mathbf{Z}}}
\newcommand{\Al}{\textup{\textsf{A}}}
\newcommand{\Sy}{\textup{\textsf{S}}}
\newcommand{\tw}[1]{{}^#1}
\renewcommand{\mod}{\bmod \,}

\begin{document}

\title[characters of even degree and normal Sylow
$2$-subgroups] {Irreducible characters of even degree\\ and normal
Sylow $2$-subgroups}

\author{Nguyen Ngoc Hung}
\address{Department of Mathematics, The University of Akron, Akron,
Ohio 44325, USA} \email{hungnguyen@uakron.edu}

\author{Pham Huu Tiep}
\address{Department of Mathematics, University of Arizona, Tucson,
Arizona 85721, USA}
\email{tiep@math.arizona.edu}

\thanks{N.\,N. Hung was partially supported by the NSA Young
  Investigator Grant \#H98230-14-1-0293 and a Faculty Scholarship
  Award from Buchtel College of Arts and Sciences, The University
  of Akron.}
\thanks{P.\,H. Tiep was partially supported by the NSF grant DMS-1201374 and the Simons
  Foundation Fellowship 305247.}



\date{\today}

\begin{abstract} The classical It\^{o}-Michler theorem on character
degrees of finite groups asserts that if the degree of every complex irreducible character of a finite group
$G$ is coprime to a given prime $p$, then $G$ has a normal Sylow $p$-subgroup. We propose a new
direction to generalize this theorem by introducing an invariant
concerning character degrees. We show that if the average degree of
linear and even-degree irreducible characters of $G$ is less than
$4/3$ then $G$ has a normal Sylow $2$-subgroup, as well as
corresponding analogues for real-valued characters and strongly real
characters. These results improve on several earlier results
concerning the It\^{o}-Michler theorem.
\end{abstract}

\maketitle


\section{Introduction}

The celebrated It\^{o}-Michler theorem \cite{Ito,Michler} is one of
the deep and fundamental results on the relation between character
degrees and local structure of finite groups. It asserts that if a
prime $p$ does not divide the degree of every complex irreducible
character of a finite group $G$, then $G$ has a normal abelian Sylow
$p$-subgroup.

Over the past decades, there have been several variations and
refinements of this result by considering Brauer characters
\cite{Manz, Manz-Wolf,Michler2}, nonvanishing elements
\cite{Dolfi-Pacifini-Sanus-Spiga, Navarro-Tiep}, fields of character
values \cite{Dolfi-Navarro-Tiep,Navarro-Tiep,Isaacs-Navarro}, and
Frobenius-Schur indicator \cite{Marinelli-Tiep,Tiep}. One primary
direction is to weaken the condition that all irreducible characters
of $G$ have degree coprime to $p$, and assume instead only that a
subset of characters with a specified field of values has this
property, see \cite{Dolfi-Navarro-Tiep,Marinelli-Tiep} for
real-valued characters and \cite{Navarro-Tiep} for $p$-rational
characters.

In this paper, we introduce an invariant concerning character
degrees and propose to generalize the It\^{o}-Michler theorem in a
completely new direction.

Given a finite group $G$, let $\Irr(G)$ denote the set of all
complex irreducible characters of $G$, then write
\[
\Irr_p(G):=\{\chi\in\Irr(G)\mid \chi(1)=1 \text { or } p\mid
\chi(1)\}\] and \[\acd_p(G):=\frac{\sum_{\chi\in\Irr_p(G)}
\chi(1)}{|\Irr_p(G)|}
\]
so that $\acd_p(G)$ is the average degree of linear characters and
irreducible characters of $G$ with degree divisible by $p$. Then the
It\^{o}-Michler theorem can be reformulated in the following way:

\begin{quote}{\emph{If $\acd_p(G)=1$ then $G$ a has normal abelian Sylow
$p$-subgroup}}.\end{quote}

Our first result significantly improves this for the prime $p=2$.

\begin{theorem}\label{theorem-main-1}
Let $G$ be a finite group. If $\acd_2(G)<4/3$ then $G$ has a normal
Sylow $2$-subgroup.
\end{theorem}

\noindent Theorem~\ref{theorem-main-1} basically says that even when
a group has some irreducible characters of even degree, it still has
a normal Sylow $2$-subgroup as long as the number of linear
characters of the group is large enough. This of course implies the
It\^{o}-Michler theorem for $p=2$ where it is required that the
group has no irreducible characters of even degree at all.

One of the key steps in the proof of Theorem~\ref{theorem-main-1} is
to establish the solvability of the groups in consideration. In
fact, we can do more.

\begin{theorem}\label{theorem-main-2}
Let $G$ be a finite group. If $\acd_2(G)<5/2$ then $G$ is solvable.
\end{theorem}

\noindent We remark that both bounds in
Theorems~\ref{theorem-main-1} and \ref{theorem-main-2} are optimal,
as shown by the groups $\Sy_3$ and $\Al_5$. Furthermore, $\acd_2(G)$
of non-abelian $2$-groups $G$ can get as close to $1$ as we wish --
just consider the extraspecial $2$-groups, and therefore one can not
get the commutativity of the Sylow $2$-subgroup in
Theorem~\ref{theorem-main-1} as in the It\^{o}-Michler theorem.

In fact, we can also improve on some main results of
\cite{Dolfi-Navarro-Tiep} and \cite{Marinelli-Tiep}, by restricting
our attention to only \emph{real-valued} characters or even
\emph{strongly real} characters. A character $\chi\in\Irr(G)$ is
called \emph{real-valued} if $\chi(g)\in\RR$ for all $g\in G$, and
{\it strongly real} if it has Frobenius-Schur indicator $1$,
equivalently, if it is afforded by a real representation. Let
\[
\Irr_{p,\RR}(G):=\{\chi\in\Irr_p(G)\mid \chi \text{ is
real-valued}\},~
\acd_{p,\RR}(G):=\frac{\sum_{\chi\in\Irr_{p,\RR}(G)}
\chi(1)}{|\Irr_{p,\RR}(G)|},
\]
and
\[
\Irr_{p,+}(G):=\{\chi\in\Irr_p(G)\mid \chi \text{ is strongly real}\},~
\acd_{p,+}(G):=\frac{\sum_{\chi\in\Irr_{p,+}(G)}
\chi(1)}{|\Irr_{p,+}(G)|}
\]

\begin{theorem}\label{theorem-main-3}
Let $G$ be a finite group. We have:
\begin{enumerate}
\item[(i)] If $\acd_{2,+}(G)\leq 2$ then $G$ is solvable.
\item[(ii)] If $\acd_{2,+}(G)<4/3$ then $G$ has a normal Sylow $2$-subgroup.
\end{enumerate}
\end{theorem}

Theorem \ref{theorem-main-3}(ii) immediately implies
\cite[Theorem~A]{Dolfi-Navarro-Tiep} and
\cite[Theorem~B]{Marinelli-Tiep}. Furthermore, since any real-valued
character of degree 1 is automatically strongly real, it has the
following consequence.

\begin{corollary}\label{corollary-main-4}
Let $G$ be a finite group. We have:
\begin{enumerate}
\item[(i)] If $\acd_{2,\RR}(G)\leq 2$ then $G$ is solvable.
\item[(ii)] If $\acd_{2,\RR}(G)<4/3$ then $G$ has a normal Sylow $2$-subgroup.
\end{enumerate}
\end{corollary}

Again, the example of $\Sy_3$ shows that the bounds in
Theorem~\ref{theorem-main-3}(ii) and
Corollary~\ref{corollary-main-4}(ii) are optimal.

To prove Theorems~\ref{theorem-main-2} and \ref{theorem-main-3}(i),
we have to use the classification of finite simple groups to show
that every nonabelian finite simple group $S$ possesses an
irreducible character of even and large enough degree which is
extendible to its stabilizer in $\Aut(S)$,
cf.~Theorem~\ref{theorem-extendible-characters-simplegroups}.
Together with Proposition~\ref{proposition-n1-n2}, this result
allows us to bound the number of (strongly real) irreducible
characters of small degrees of a finite group with a nonabelian
minimal normal subgroup, and then to control the invariant
$\acd_2(G)$ of such a group, see Section~\ref{section-preliminary}
and Proposition~\ref{proposition-5/2}. We hope that the techniques
developed here will be useful in the future study of other problems
involving the average degree of a certain set of characters and
other invariants concerning character degrees like the largest
character degree \cite{Cossey-Halasi-Maroti-Nguyen,Gluck,HHN,HLS} or
the character degree ratio \cite{Cossey-Nguyen,Lewis-Nguyen}.

One obvious question that one may ask is: is there an analogue of
Theorem~\ref{theorem-main-1} for odd primes? Although our ideas in
the proof for the prime $2$ do not carry out smoothly to odd primes,
we believe that the following is true.

\begin{conjecture}\label{conjecture}
Let $p$ be a prime and $G$ a finite group. If $\acd_p(G)<2p/(p+1)$
then $G$ has a normal Sylow $p$-subgroup.
\end{conjecture}

\noindent The bound in Conjecture~\ref{conjecture} perhaps is not
optimal for all primes. If $C_p$ (the cyclic group of order $p$) can
act nontrivially on an abelian group of order $p+1$, then the bound
clearly can not be lower (for instance $p=3,7$, or any Mersenne
prime). But when $p=5$ for example, $C_5$ can only act trivially on
an abelian group of order $6$, and we think that the best possible
bound is not $10/6$.

Theorems \ref{theorem-main-2}, \ref{theorem-main-1}, and
\ref{theorem-main-3} are respectively proved in
Sections~\ref{section-theorem-1.2}, \ref{section-theorem-1.1}, and
\ref{section-strongly-real}.

\section{Extendibility of some characters of even degree}\label{section-preliminary}

Throughout the paper, for a finite group $G$ and a positive integer
$k$, we write $n_k(G)$ to denote the number of irreducible complex
characters of $G$ of degree $k$, and $n_{k,+}(G)$ to denote the
number of strongly real, irreducible complex characters of $G$ of
degree $k$. Furthermore, if $N$ is a normal subgroup of $G$, then
$n_k(G|N)$ denotes the number of irreducible characters of $G$ of
degree $k$ whose kernels do not contain $N$, and similarly for
$n_{k,+}(G|N)$. If $\theta$ is a character of a normal subgroup of
$G$, we write $I_G(\theta)$ to denote the stabilizer or the inertia
subgroup of $\theta$ in $G$. Other notation is standard (and follows
\cite{Isaacs1}) or will be defined when needed.

We need the following result, whose proof relies on the
classification of finite simple groups.

\begin{theorem}\label{theorem-extendible-characters-simplegroups}
Every nonabelian finite simple group $S$ has an irreducible
character $\theta$ of even degree such that $\theta(1)\geq 4$ and
$\theta$ is extendible to a strongly real character of $I_{\Aut(S)}(\theta)$. Furthermore, if
$S\ncong \Al_5$ then $\theta$ can be chosen so that $\theta(1)\geq
8$.
\end{theorem}

\begin{proof}
(i) The cases where $S \cong \Al_n$ with $5 \leq n \leq 8$, or $S
\cong \PSL_2(q)$ with $q \leq 19$, or $S \cong \PSU_3(3)$,
$\PSp_4(3)$, $\Sp_6(2)$, $\tw2 F_4(2)'$, or $S$ is one of the $26$
sporadic finite simple groups, can be checked directly using
\cite{Atl1}. (We note that in all these cases but $S = \Al_6$, we
can always find $\theta$ so that it has a rational-valued extension
to $I_{\Aut(S)}(\theta)$.) In what follows, we may therefore assume
that $S$ is not isomorphic to any of the listed groups. In
particular, it follows from the main results of \cite{R, SZ} that
the degree of any nontrivial complex irreducible character of $S$ is
at least $8$.

\smallskip
(ii) Certainly, one can find many different choices for the desired character $\theta$. In what follows, having
in mind some other applications, we will try to construct $\theta$ in such a way that its extension to $I_{\Aut(S)}(\theta)$ is
rational-valued if possible.

Assume now that $S \cong \Al_n$ with $n \geq 9$.  Consider the
irreducible characters $\alpha, \beta \in \Irr(\Sy_n)$ labeled by
the partitions $(n-2,2)$ and $(n-2,1^2)$, of degree $n(n-3)/2$ and
$(n-1)(n-2)/2$ respectively. Since the given partitions are not
self-conjugate, $\alpha$ and $\beta$ both restrict irreducibly to
$S$. Furthermore, $\beta(1) = \alpha(1)+1 \geq 28$, and so exactly
one of $\alpha$, $\beta$ has even degree. As $\Aut(S) \cong \Sy_n$,
we are done in this case by choosing $\theta \in
\{\alpha_S,\beta_S\}$ of even degree.

Next we consider the case $S$ is a finite simple group of Lie type in characteristic $2$. As shown in
\cite{F}, the Steinberg character of $S$, of degree $|S|_2$, extends to a character of a rational representation
of $\Aut(S)$, whence we are done again.

\smallskip
(iii) From now on we may assume that $S$ is a finite simple group of Lie type, defined over a finite field
$\FF_q$ of odd characteristic $p$.

Consider the cases where $S = \PSU_n(q)$ with $n \geq 4$,
$\PSp_{2n}(q)$ with $n \geq 2$, $\Omega_{2n+1}(q)$ with $n \geq 3$,
$\mathrm{P}\Omega^{+}_{2n}(q)$ with $n \geq 5$, or
$\mathrm{P}\Omega^{-}_{2n}(q)$ with $n \geq 4$. As shown in pp.~1--4
of the proof of \cite[Theorem~2.1]{Dolfi-Navarro-Tiep}, $\Aut(S)$
has a rank $3$ permutation character $\rho = 1+\alpha+\beta$, where
the characters $\alpha$ and $\beta$ both restrict nontrivially and
irreducibly to $S$. Furthermore, exactly one of $\alpha$ and $\beta$
is of even degree (and both are afforded by rational representations). Hence we are done
by choosing $\theta \in \{\alpha_S,\beta_S\}$ of even degree.

The same argument, but applied to a rank $5$ permutation character
of $\Aut(S)$, see p.~5 of the proof of
\cite[Theorem~2.1]{Dolfi-Navarro-Tiep}, handles the cases $S =
E_6(q)$ or $\tw2 E_6(q)$.

Suppose that $S = \PSL_n(q)$ with $n \geq 4$. As shown in the proof
of \cite[Proposition~5.5]{NT1}, $\Aut(S)$ has a permutation representation, whose
character contains a rational-valued irreducible character $\gamma$ of (even) degree
$$\left\{ \begin{array}{ll}(q^n-1)(q^{n-1}-1)/(q-1)^2, & q \equiv 1 (\mod 4),\\
     (q^n-1)(q^{n-1}-1)/(q^2-1), & q \equiv 3 (\mod 4), \end{array} \right.$$
with multiplicity one
that restricts irreducibly to $S$. It follows that $\gamma$ is also afforded by
a rational representation, and we can choose $\theta = \gamma_S$.

\smallskip
(iv) In the remaining cases, our choice of $\theta$ yields a not
necessarily rational, but still admits a strongly real extension to $J := I_{\Aut(S)}(\theta)$.
First, if $S = \PSL_3(q)$, then the (unique) character $\theta \in
\Irr(S)$ of degree $q^2+q$ extends to $\Aut(S)$ by
\cite[Lemma~6.2]{Tiep}. Next, if $S = G_2(q)$, $\tw2 G_2(q)$, $\tw3
D_4(q)$, $F_4(q)$, or $E_8(q)$, then the proof of
\cite[Proposition~4.4]{Marinelli-Tiep} yields a strongly real
character $\theta \in \Irr(S)$ of even degree, with $J = S$.

If $S = \mathrm{P}\Omega^+_8(3)$ then by p.~1 of the proof of
\cite[Proposition~4.9]{Marinelli-Tiep} we can choose $\theta$ of
degree $300$.

Suppose now that $S = \PSL_2(q)$, $\PSU_3(q)$, or (a simple group of
type) $D_4(q)$ with $q \geq 5$, or $E_7(q)$. Then we can view $S$ as
the derived subgroup of a finite Lie-type group $G$ of adjoint type:
$G = \PGL_2(q)$, $\PGU_3(q)$, $D_4(q)_{{\rm ad}}$, or $E_7(q)_{{\rm
ad}}$, respectively. As shown in the proof of
\cite[Proposition~4.5]{Marinelli-Tiep} (for types $A_1$ and $E_7$),
Case IIb of the proof of \cite[Proposition~4.7]{Marinelli-Tiep} (for
$\PSU_3$), and the proof of \cite[Proposition~4.9]{Marinelli-Tiep}
(for type $D_4$), $G$ contains a strongly real character of even
degree that restricts to an irreducible character $\theta \in
\Irr(S)$ with $J = G$. So we are done in these cases as well.
\end{proof}

When a finite group $G$ has a nonabelian minimal normal subgroup
$N$, by using
Theorem~\ref{theorem-extendible-characters-simplegroups}, we can
produce an irreducible character of $N$ of even degree that is
extendible to (a strongly real character of) its stabilizer in $G$.

\begin{theorem}\label{theorem-extendible-characters-G}
Let $G$ be a finite group with a nonabelian minimal normal subgroup
$N\ncong\Al_5$. Then there exists $\varphi\in\Irr(N)$ of even degree
such that $\varphi(1)\geq 8$ and $\varphi$ is extendible to a strongly real
character of $I_G(\varphi)$.
\end{theorem}

\begin{proof}
Since $N$ is a nonabelian minimal normal subgroup of $G$, we have
$N\cong S\times S\times\cdots\times S$, a direct product of $k$
copies of a nonabelian simple group $S$. Replacing $G$ by
$G/\bC_G(N)$ if necessary, we may assume that $\bC_G(N)=1$. Then we
have
\[
N\unlhd G \leq \Aut(N)=\Aut(S)\wr \Sy_k.
\]
Let $\theta$ be an irreducible character of $S$ found in
Theorem~\ref{theorem-extendible-characters-simplegroups} and let
$\mathcal{O}_S$ be the orbit of $\theta$ in the action of $\Aut(S)$
on $\Irr(S)$. Consider the character
$\varphi:=\theta\times\cdots\times\theta\in\Irr(N)$. Then the orbit
of $\varphi$ under the action of $\Aut(N)$ is
\[
\mathcal{O}_N:=\{\theta_1\times\theta_2\times\cdots\times\theta_k
\in\Irr(N)\mid \theta_i\in \mathcal{O}_S\}.
\]
Clearly $\varphi$ is invariant under $I_{\Aut(S)}(\theta)\wr \Sy_k$.
On the other hand,
\[
|\Aut(N):I_{\Aut(S)}(\theta)\wr
\Sy_k|=|\Aut(S):I_{\Aut(S)}(\theta)|^k=|\mathcal{O}_S|^k=|\mathcal{O}_N|.
\]
Therefore we deduce that
$I_{\Aut(N)}(\varphi)=I_{\Aut(S)}(\theta)\wr \Sy_k$.

Assume that $\theta$ extends to a strongly real character $\alpha$ of
$J := I_{\Aut(S)}(\theta)$, say afforded by an $\RR J$-module $V$.
Then $I_{\Aut(N)}(\varphi)$ acts naturally on the $\RR$-space $V^{\otimes k}$, on which
$N$ acts with character $\varphi$. It follows that $\varphi$ is extendible to the strongly
real character of $I_G(\varphi)=G\cap I_{\Aut(N)}(\varphi)$ afforded by $V^{\otimes k}$.

Since $\theta(1)\geq 4$ in general and $\theta(1)\geq 8$ when
$S\ncong \Al_5$, we observe that $\varphi(1)=\theta(1)^k\geq 8$ as
long as $(S,k)\neq (\Al_5,1)$, and we are done.
\end{proof}

Theorem~\ref{theorem-extendible-characters-G} can be used to bound
the number of irreducible characters of degrees 1 and 2 of finite
groups with a nonabelian minimal normal subgroup, as shown in the
next proposition. This should be compared with Proposition 3.2 of
\cite{Hung}.

\begin{proposition}\label{proposition-n1-n2}
Let $G$ be a finite group with a nonabelian minimal normal subgroup
$N$. Assume that there is some $\varphi\in\Irr(N)$ such that
$\varphi$ is extendible to $I_G(\varphi)$. Let
$d:=\varphi(1)|G:I_G(\varphi)|$. Then the following hold.
\begin{enumerate}
\item[(i)] $n_1(G)\leq n_d(G)|G:I_G(\varphi)|$.
\item[(ii)] If $\varphi$ extends to a strongly real character of $I_G(\varphi)$, then
$$n_{1,+}(G)\leq n_{d,+}(G)|G:I_G(\varphi)|.$$
\item[(iii)] $n_2(G)\leq
n_{2d}(G)|G:I_G(\varphi)|+\frac{1}{2}n_d(G)|G:I_G(\varphi)|$.
\end{enumerate}
Moreover, if $\varphi$ is invariant in $G$ then $n_2(G)\leq
n_{2d}(G)$.
\end{proposition}

\begin{proof} For simplicity we write $I:=I_G(\varphi)$.

\medskip

(i) First, since $n_1(G)=|G:G'|$ and $n_1(I)=|I:I'|$, we have
$n_1(G)\leq |G:I|n_1(I)$. Therefore, we wish to show that
$n_1(I)\leq n_d(G)$ where $d:=\varphi(1)|G:I|$.

As $N=N'\subseteq I'$, the normal subgroup $N$ is contained in the
kernel of every linear character of $I$ so that \[n_1(I)=n_1(I/N).\]
Recall that $\varphi\in\Irr(N)$ is extendible to $I$ and so we let
$\chi\in\Irr(I)$ be an extension of $\varphi$. Using Gallagher's
theorem and Clifford's theorem (see \cite[Corollary~6.17 and
Theorem~6.11]{Isaacs1}), we see that each linear character $\lambda$
of $I/N$ produces the irreducible character $\lambda\chi$ of $T$ of
degree $\varphi(1)$, and this character in turns produces the
irreducible character $(\lambda\chi)^G$ of $G$ of degree
$(\lambda\chi)^G(1)=\varphi(1)|G:I|=d$. As the maps $\lambda\mapsto
\lambda \chi\mapsto (\lambda\chi)^G$ are both injective, it follows
that
\[n_1(I/N)\leq n_{d}(G)\] and we therefore have
$n_1(I)\leq n_d(G)$, as desired.

\medskip
(ii) For any group $X$, let $X^*$ denote the subgroup generated by
all $x^2$, $x \in X$. Then note that $n_{1,+}(G) = |G:G^*|$,
$n_{1,+}(I) = |I:I^*|$, and $I^* \leq G^*$, whence $n_{1,+}(G)\leq
|G:I|n_{1,+}(I)$. Furthermore, for any $\chi \in \Irr(G)$ and any
strongly real linear character $\lambda$ of $G$, $\lambda^2 = 1_G$
and so $\chi$ and $\chi\lambda$ have the same Frobenius-Schur
indicator. In particular, $\chi\lambda$ is strongly real if and only
if $\chi$ is. Furthermore, if $\rho$ is a strongly real character of
a subgroup $T \leq G$, then so is the induced character $\rho^G$.
Now we can argue as in (i) to complete the proof.

\medskip
(iii) We first claim that
\[
n_2(G)\leq n_2(I)|G:I|+\frac{1}{2}n_1(I)|G:I|.
\]
Let $\chi\in\Irr(G)$ with $\chi(1)=2$. Take $\phi$ to be an
irreducible constituent of $\chi\downarrow_I$. Frobenius reciprocity
then implies that $\chi$ in turn is an irreducible constituent of
$\phi^G$. If $\phi(1)=2$ then as $\phi^G(1)=2|G:I|$, there are at
most $|G:I|$ irreducible constituents of degree $2$ of $\phi^G$. We
deduce that there are at most $n_2(I)|G:I|$ irreducible characters
of degree $2$ of $G$ that arise as constituents of $\phi^G$ with
$\phi(1)=2$. On the other hand, if $\phi(1)=1$ then, as
$\phi^G(1)=|G:I|$, there are at most $|G:I|/2$ irreducible
constituents of degree $2$ of $\phi^G$. As above, we deduce that
there are at most $n_1(I)|G:I|/2$ irreducible characters of degree
$2$ of $G$ that arise as constituents of $\phi^G$ with $\phi(1)=1$.
So the claim is proved.

Since we have already proved in (i) that $n_1(I)=n_1(I/N)\leq
n_d(G)$, to prove Proposition \ref{proposition-n1-n2}(iii) it suffices to show that $n_2(I)\leq
n_{2d}(G)$.

We claim that \[n_2(I)=n_2(I/N)\] or in other words, $N$ is
contained in the kernel of every irreducible character of degree $2$
of $I$. Let $\phi\in\Irr(I)$ with $\phi(1)=2$. Since $N$ has no
irreducible character of degree 2 and has only one linear character,
which is the trivial one, it follows that $\phi_N=2\cdot1_N$. We
then have $N\subseteq \Ker(\phi)$, as claimed.

Recall that $\chi\in\Irr(I)$ is an extension of $\psi$. Using
Gallagher's theorem and Clifford's theorem again, we obtain that
each irreducible character $\mu\in\Irr(I/N)$ of degree $2$ produces
the character $(\mu\chi)^G\in\Irr(G)$ of degree
$(\mu\chi)^G(1)=2\psi(1)|G:I|=2d$. It follows that
\[n_2(I/N)\leq n_{2d}(G),\] and thus $n_2(I)\leq n_{2d}(G)$, as desired.

If $\varphi$ is invariant in $G$ then $G=I$,  yielding
immediately that $n_2(G)\leq n_{2d}(G)$.
\end{proof}


\section{Solvability -
Theorem~\ref{theorem-main-2}}\label{section-theorem-1.2}

In this section, we use the results in
Section~\ref{section-preliminary} to prove
Theorem~\ref{theorem-main-2}. The next proposition handles an
important case of this theorem.

\begin{proposition}\label{proposition-5/2}
Let $G$ be a finite group with a nonabelian minimal normal subgroup.
Then $\acd_2(G)\geq 5/2$.
\end{proposition}

\begin{proof}
Let $N$ be a nonabelian minimal normal subgroup of $G$. First we
assume that $N\ncong \Al_5$. Then
Theorem~\ref{theorem-extendible-characters-G} guarantees that there
is $\varphi\in\Irr(N)$ of even degree such that $\varphi(1)\geq 8$
and $\varphi$ is extendible to the inertia subgroup $I_G(\varphi)$.
Using Proposition~\ref{proposition-n1-n2}, we then have
\[
n_1(G)\leq n_d(G)|G:I_G(\varphi)|
\]
and
\[n_2(G)\leq
n_{2d}(G)|G:I_G(\varphi)|+\frac{1}{2}n_d(G)|G:I_G(\varphi)|,
\]
where
$d:=\varphi(1)|G:I_G(\varphi)|$. It follows that
\[
3n_1(G)+n_2(G)\leq
\frac{7}{2}n_d(G)|G:I_G(\varphi)|+n_{2d}(G)|G:I_G(\varphi)|.
\]
Since $\varphi(1)\geq 8$, we have $d\geq 8|G:I_G(\varphi)|\geq 8$
and hence we can check that $(7/2)|G:I_G(\varphi)|\leq (7/16)d<
2d-5$ and $|G:I_G(\varphi)|\leq d/8<4d-5$. It follows that
\[
3n_1(G)+n_2(G)<\sum_{2|k\geq 4}(2k-5)n_k(G),
\]
and thus \[\sum_{2|k \text{ or } k=1}(2k-5)n_k(G)>0.\] Therefore
\[2\sum_{2|k \text{ or } k=1} kn_k(G)> 5\sum_{2|k \text{ or }
k=1}n_k(G),\] and we have $\acd_2(G)>5/2$, as desired.

So it remains to consider $N\cong \Al_5$. Then $N$ has an
irreducible character $\varphi$ of degree $4$ that is extendible to
$\Aut(N)$ (see \cite[p.~2]{Atl1}), and hence extendible to $G$ as
well. It follows from Proposition~\ref{proposition-n1-n2} that
$n_1(G)\leq n_4(G)$ and $n_2(G)\leq n_{8}(G)$. Thus
\[
3n_1(G)+n_2(G)\leq 3n_4(G)+n_8(G)\leq \sum_{2|k\geq 4}(2k-5)n_k(G).
\]
Again this yields $\acd_2(G)\geq5/2$ and the proof is completed.
\end{proof}

We are now ready to prove Theorem~\ref{theorem-main-2}, which we
restate below.

\begin{theorem}\label{theorem-main-2-again}
Let $G$ be a finite group. If $\acd_2(G)<5/2$ then $G$ is solvable.
\end{theorem}

\begin{proof}
Assume, to the contrary, that the theorem is false, and let $G$ be a
minimal counterexample. In particular $\acd_2(G)<5/2$ and $G$ is
nonsolvable. Let $H\trianglelefteq G$ be minimal such that $H$ is
nonsolvable. Then clearly $H$ is perfect and contained in the last
term of the derived series of $G$. Choose a minimal normal subgroup
$N$ of $G$ such that $N\subseteq H$, and when $[H,\bO_\infty(H)]\neq
1$ we choose $N\subseteq [H,\bO_\infty(H)]$, where $\bO_\infty(H)$
denotes the largest normal solvable subgroup of $H$.

In view of Proposition~\ref{proposition-5/2}, we can assume that $N$
is abelian. It follows that the quotient $G/N$ is nonsolvable since
$G$ is nonsolvable. By the minimality of $G$, we must have
$\acd_2(G/N)\geq 5/2$. So
\[\acd_2(G)<5/2\leq\acd_2(G/N).\]
Since $n_k(G/N)\leq n_k(G)$ for every positive integer $k$ and
$n_1(G/N)=n_1(G)$ as $N\subseteq G'$, it follows that
$n_2(G/N)<n_2(G)$ and thus there exists $\chi\in\Irr(G)$ such that
$\chi(1)=2$ and $N\nsubseteq \Ker(\chi)$.

It then has been shown in the proof of
\cite[Theorem~2.2]{Isaacs-Loukaki-Moreto} that $G=LC$ is a central
product with the amalgamated subgroup $N=\Center(L)$ of order 2,
where \[L=\SL_2(5) \text{ and }
C/\Ker(\chi):=\Center(G/\Ker(\chi)).\] Since $G=LC$ is a central
product with the central amalgamated subgroup $N$, there is a
bijection $(\alpha,\beta)\mapsto \tau$ from
$\Irr(L|N)\times\Irr(C|N)$ to $\Irr(G|N)$ such that
$\tau(1)=\alpha(1)\beta(1)$. If $(\alpha,\beta)\mapsto \chi$ under
the above bijection, we must have $\beta(1)=1$ since
$L\cong\SL(2,5)$ and there are only three possibilities for
$\alpha(1)$, namely 2, 4, and 6. So $\beta\in\Irr(C|N)$ is an
extension of the unique non-principal linear character of $N$. Using
Gallagher's theorem, we then have a degree-preserving bijection from
$\Irr(C/N)$ to $\Irr(C|N)$. In particular, $n_1(C|N)=n_1(C/N).$
Since $G/L\cong C/N$ and $L\subseteq G'$, we have
\[n_1(C|N)=n_1(C/N)=n_1(G).\]

Employing the arguments in the proof of
\cite[Theorem~B]{Moreto-Nguyen}, we can evaluate and estimate
$n_2(G)$, $n_4(G)$, and $n_6(G)$ in terms of $n_1(G)$ as follows:
\begin{enumerate}
\item[(i)] $n_2(G)=2n_1(G)+n_2(C/N)$,
\item[(ii)] $n_4(G)\geq 2n_1(G)$, and
\item[(iii)] $n_6(G)\geq n_1(G)+2n_2(C/N)$.
\end{enumerate}
Now putting all things together, we have
\begin{align*}\sum_{2|k\geq 4}(2k-5)n_k(G)&\geq 3n_4(G)+7n_6(G)\\
&\geq 6n_1(G)+7(n_1(G)+2n_2(C/N)\\
&=13n_1(G)+14n_2(C/N)\\
&>3n_1(G)+n_2(G).
\end{align*}
It then follows that $\acd_2(G)>5/2$ and this is a contradiction.
\end{proof}


\section{Normal Sylow $2$-subgroups -
Theorem~\ref{theorem-main-1}}\label{section-theorem-1.1}

The next lemma is crucial in the proof of
Theorem~\ref{theorem-main-1}.

\begin{lemma}\label{lemma-orbit}
Let $G=N\rtimes M$ where $N$ is an abelian group. Assume that no
nonprincipal irreducible character of $N$ is invariant under $M$. If
$\acd_2(G)<4/3$ then there is no orbit of even size in the action of
$M$ on the set of irreducible characters of $N$.
\end{lemma}

\begin{proof}
Let $\{1_N=\alpha_0,\alpha_1,\ldots,\alpha_f\}$ be a set of
representatives of the action of $M$ on $\Irr(N)$. For each $1\leq
i\leq f$, let $I_i$ be the inertia subgroup of $\alpha_i$ in $G$.
Since  no nonprincipal irreducible character of $N$ is invariant
under $M$, we observe that every $I_i$ is a proper subgroup of $G$.

Assume, to the contrary, that there is some orbit of even size in
the action of $M$ on $\Irr(N)$. Then there exists some index $1\leq
i\leq f$ such that $|G:I_i|$ is even. For $0\leq i\leq f$, set
$$n_{i,1}:=n_1(I_i/N),~~n_{i,\text{even}}:=\sum_{2|k} n_{k}(I_i/N),~~s_{i, \text{even}}:=\sum_{\lambda\in\Irr(I_i/N),~2|\lambda(1)} \lambda(1).$$
Since $G$ splits over $N$, it is clear that every $I_i$ also splits
over $N$. It follows that $\alpha_i$ extends to a linear character,
say $\beta_i$, of $I_i$ as $\alpha_i$ is linear. Gallagher's theorem
then implies that the mapping $\lambda\mapsto \lambda\beta_i$ is a
bijection from $\Irr(I_i/N)$ to the set of irreducible characters of
$I_i$ lying above $\alpha_i$. Using Clifford's theorem, we then
obtain a bijection $\lambda\mapsto (\lambda\beta_i)^G$ from
$\Irr(I_i/N)$ to the set of irreducible characters of $G$ lying
above $\alpha_i$. We note that
$(\lambda\beta_i)^G(1)=|G:I_i|\lambda(1)$ and hence
$(\lambda\beta_i)^G(1)$ is even if and only if either $|G:I_i|$ is
even or $\lambda(1)$ is even.

We have
\begin{align*}
\sum_{\chi\in\Irr(G), \chi(1)=1 \text{ or even}}
\chi(1)&=n_1(G/N)+\sum_{|G:I_i| \text{ even}}
|G:I_i|n_{i,1}+\sum_{i=0}^f |G:I_i|s_{i,\text{even}}\\
&\geq n_1(G/N)+\sum_{|G:I_i| \text{ even}}
|G:I_i|n_{i,1}+2\sum_{i=0}^f |G:I_i|n_{i,\text{even}}.
\end{align*}
On the other hand,
\[
\sum_{\chi\in\Irr(G), \chi(1)=1 \text{ or even}}
\chi(1)=\acd_2(G)\left(n_1(G/N)+\sum_{|G:I_i| \text{ even}}
n_{i,1}+\sum_{i=0}^f n_{i,\text{even}}\right).
\]
Therefore, we deduce that
\begin{align*}
\sum_{|G:I_i| \text{ even}} (|G:I_i|-\acd_2(G))n_{i,1}&+\sum_{i=0}^f
(2|G:I_i|-\acd_2(G))n_{i,\text{even}}\\
&\leq (\acd_2(G)-1)n_1(G/N).
\end{align*}
Since $\acd_2(G)<4/3$ and $|G:I_i|\geq1$ for every $0\leq i\leq f$,
it follows that
\[
\sum_{|G:I_i| \text{ even}} (|G:I_i|-\acd_2(G))n_{i,1}\leq
(\acd_2(G)-1)n_1(G/N)
\]
and hence
\[
(|G:I_j|-\acd_2(G))n_{j,1}\leq (\acd_2(G)-1)n_1(G/N)
\]
for some $1\leq j\leq f$ such that $|G:I_j|$ is even.

Observe that $n_1(G/N)=|(G/N):(G/N)'|$ and
$n_{j,1}=n_1(I_j/N)=|(I_j/N):(I_j/N)'|$. Therefore $n_1(G/N)\leq
|G:I_j|n_{j,1}$. It follows from the above inequality that
\[
(|G:I_j|-\acd_2(G))n_{j,1}\leq (\acd_2(G)-1)|G:I_j|n_{j,1},
\]
and thus
\[
|G:I_j|\leq \frac{\acd_2(G)}{2-\acd_2(G)}.
\]
This is impossible since $|G:I_j|\geq 2$ and $\acd_2(G)<4/3$, and
the proof is complete.
\end{proof}

We now prove the main Theorem~\ref{theorem-main-1}, which is
restated below.

\begin{theorem}\label{theorem-main-1-again}
Let $G$ be a finite group. If $\acd_2(G)<4/3$ then $G$ has a normal
Sylow $2$-subgroup.
\end{theorem}

\begin{proof}
We will argue by induction on the order of $G$. We have
$\acd_2(G)<4/3$ and therefore by Theorem~\ref{theorem-main-2}, $G$
is solvable. If $G$ is abelian then there is nothing to prove, so we
assume that $G$ is nonabelian. We then choose a minimal normal
subgroup $N$ of $G$ such that $N\subseteq G'$. As $G$ is solvable,
we have that $N$ is abelian.

Since $N\subseteq G'$, if an irreducible character of $G$ has kernel
not containing $N$, its degree must be at least 2. We therefore
deduce that $\acd_2(G/N)\leq \acd_2(G)<4/3$ and it follows from the
induction hypothesis that $G/N$ has a normal Sylow $2$-subgroup, say
$Q/N$.

If $N$ is a $2$-group then $Q$ is a normal Sylow $2$-subgroup of
$G$, and we are done. So we assume from now on that $N$ is an
elementary abelian group of odd order. The Schur-Zassenhaus theorem
then implies that $Q=N\rtimes P$ where $P$ is a Sylow $2$-subgroup
of $Q$ (and $G$ as well). By Frattini's argument, we have
$G=Q\bN_G(P)=N\bN_G(P)$. If $N\subseteq \bN_G(P)$ then it follows
that $G=\bN_G(P)$ and we are done. So we can assume that
$N\nsubseteq \bN_G(P)$. This implies that $N\cap \bN_G(P)<N$. As $N$
is abelian, $N\cap \bN_G(P)$ is a normal subgroup of $G$, and we
deduce that $N\cap \bN_G(P)=1$ by the minimality of $N$. We conclude
that $G=N\rtimes \bN_G(P)$.


If $N\subseteq \bZ(G)$, then we would have $Q=N\times P$, so that $P
\unlhd G$, and we are done again. So we assume that $N$ is
noncentral in $G$. Thus, by the minimality of $N$, we have
$[N,G]=N$. It follows that no nonprincipal irreducible character of
$N$ is invariant under $\bN_G(P)$.

We are now in the situation of Lemma~\ref{lemma-orbit}, and
therefore we conclude that there is no orbit of even size in the
action of $\bN_G(P)$ on the set of irreducible characters of $N$. In
particular, there is no orbit of even size in the action of $P$ on
the set of irreducible characters of $N$. This means that $P$ acts
trivially on $N$ since $P$ is a $2$-group. Now we have $Q=N\times
P\unlhd G$ and, as $N$ has odd order, we deduce that $P\unlhd G$ and
this completes the proof of the theorem.
\end{proof}


\section{Strongly real characters -
Theorem~\ref{theorem-main-3}}\label{section-strongly-real}

In this section we prove Theorem~\ref{theorem-main-3}. We first
prove Theorem~\ref{theorem-main-3}(i).

\begin{theorem}\label{theorem-strongly-real-solvable}
Let $G$ be a finite group. If $\acd_{2,+}(G)\leq 2$ then $G$ is
solvable.
\end{theorem}

\begin{proof}
Since there is nothing to prove if $G$ is abelian, we assume that
$G'$ is non-trivial. Let $N\subseteq G'$ be a minimal normal
subgroup of $G$. If a strongly real character $\chi \in \Irr(G)$ has
kernel not containing $N$, its degree must be at least $2$, which is
above the average $\acd_{2,+}(G)$. We therefore deduce that
\[\acd_{2,+}(G/N)\leq \acd_{2,+}(G)\leq 2.\]
Working by induction, we now can assume that $N$ is nonabelian.

By Theorem~\ref{theorem-extendible-characters-G}, there exists
$\varphi\in\Irr(N)$ of even degree such that $\varphi(1)\geq 4$ and
$\varphi$ extends to a strongly real character of $I_G(\varphi)$.
(Note that we assumed $N\ncong \Al_5$ in
Theorem~\ref{theorem-extendible-characters-G}, but when $N\cong
\Al_5$ one can choose $\varphi$ to be the irreducible character of
degree $4$, and this character is extendible to a strongly real
character of $\Sy_5$.) We then apply
Proposition~\ref{proposition-n1-n2} to have
\[
n_{1,+}(G)\leq n_{d,+}(G)|G:I_G(\varphi)|,
\]
where $d:=\varphi(1)|G:I_G(\varphi)|$. Since $\varphi(1)\geq 4$, it
follows that \[|G:I_G(\varphi)| < 4|G:I_G(\varphi)|-2\leq d-2,\] and
so
\[
n_{1,+}(G) < (d-2)n_{d,+}(G)\leq \sum_{2|k} (k-2)n_{k,+}(G).
\]
Therefore we have
$$\acd_{2,+}(G) = \frac{n_{1,+}(G) + \sum_{2|k}kn_{k,+}(G)}{n_{1,+}(G) + \sum_{2|k}n_{k,+}(G)}
    > 2,$$
and this completes the proof.
\end{proof}

To prove Theorem~\ref{theorem-main-3}(ii), we begin with two known
observations on strongly real characters.

\begin{lemma}\label{lemma-strong-real}
Let $N$ be a normal subgroup of a finite group $G$ such that $G/N$
has odd order. Then every strongly real character
$\varphi\in\Irr(N)$ lies under a unique strongly real irreducible
character of $G$.
\end{lemma}

\begin{proof}
This is \cite[Lemma~2.1(ii)]{Marinelli-Tiep}.
\end{proof}

\begin{lemma}\label{lemma-strong-real-2}
Let $N\lhd G$ be such that $G/N$ is a $2$-group. Assume that
$\varphi\in\Irr(N)$ has $2$-defect $0$, and that
$\overline{\varphi}$ is $G$-conjugate to $\varphi$. Then there
exists a strongly real character $\chi\in\Irr(G)$ such that
$[\chi_N,\varphi]_N=1$.
\end{lemma}

\begin{proof}
This is \cite[Lemma~2.5]{Tiep}.
\end{proof}

We also need the following observation.

\begin{lemma}\label{lemma-strongly-real-3}
Let $P$ be a 2-group acting on an abelian group $N$ of odd order
such that $N$ is the unique minimal normal subgroup of $N \rtimes P$. Then
$|N|-1\geq |P:\Phi(P)|$, where $\Phi(P)$ is the Frattini subgroup of
$P$.
\end{lemma}

\begin{proof}
Suppose that $N$ is a product of $n$ copies of the cyclic groups
$C_q$, where $q$ is a prime. Since $N$ is the only minimal normal
subgroup of $N \rtimes P$, we have $\bC_P(N)=1$, and the action of $P$ on
$N$ induces a faithful irreducible representation $\mathfrak{X}$ of
$P$ over the field $\FF_q$. We extend this
representation to the representation
$\mathfrak{X}^{\overline{\FF}_q}$ over the algebraically closure
$\overline{\FF}_q$. Since $|P|$ is even and $q$ is odd, by Maschke's
theorem, $\mathfrak{X}^{\overline{\FF}_q}$ is completely reducible,
and moreover, is faithful since $\mathfrak{X}$ is faithful. Using
the Fong-Swan theorem \cite[Theorem~10.1]{Navarro} on lifts of
irreducible Brauer characters in solvable groups, we conclude that
$P$ has a complex faithful character, say $\chi$, of degree $n$.

Now we apply \cite[Theorem~A]{Isaacs2} to deduce that the number of
generators in a minimal generating set for $P$, say $d(P)$, is at
most $(3/2)(n-s)+s$, where $s$ is the number of linear constituents of
$\chi$. In particular, $d(P)\leq [3n/2]$, and it follows that
\[|P:\Phi(P)|\leq 2^{[3n/2]}.\]
As it is easy to check that $2^{[3n/2]}\leq 3^n-1$ for every
positive integer $n$, we then have
\[|P:\Phi(P)|\leq 3^n-1\leq q^n-1=|N|-1,\]
and the lemma follows.
\end{proof}

Lemmas~\ref{lemma-strong-real-2} and \ref{lemma-strongly-real-3}
allow us to control $\acd_{2,+}(G)$ in the following special
situation.

\begin{proposition}\label{proposition-strongly-real}
Let $G=N\rtimes P$ be a split extension of an elementary abelian
group of odd order $N$ by a nontrivial $2$-group $P$. Assume that $N$ is the
unique minimal normal subgroup of $G$. Then $\acd_{2,+}(G)\geq 4/3$.
\end{proposition}

\begin{proof}
It is well known that groups of odd order have no non-principal real
irreducible character. Therefore $1_N$ is the only real irreducible
character of $N$. It follows that every (strongly) real linear
character of $G$ has $N$ inside its kernel. In particular,
\[n_{1,+}(G)=n_{1,+}(G/N)=n_{1,+}(P) = |P:\Phi(P)|.\]
Together with Lemma~\ref{lemma-strongly-real-3}, this implies that
\[n_{1,+}(G)\leq |N|-1.\]
Hence, we can find $s$ disjoint $P$-orbits $\Omega_{1}, \ldots ,\Omega_s$ on $\Irr(N) \smallsetminus \{1_N\}$
of size $k_1, \ldots ,k_s$, such that
$$k_1 + \ldots +k_s \geq n_{1,+}(G).$$

From the hypotheses, we know that $P$ acts irreducibly and
faithfully on $N$, and therefore on $\Irr(N)$ as well since the two
actions are permutationally isomorphic. Therefore, each nontrivial
character in $\Irr(N)$ is inverted by a central involution of $P$.
In other words, for any nontrivial $\varphi\in\Irr(N)$, $\varphi$
and $\overline{\varphi}$ are $P$-conjugate.
Lemma~\ref{lemma-strong-real-2} then guarantees that there is a
strongly real character
$\chi_{\{\varphi,\overline{\varphi}\}}\in\Irr(G)$ lying above both
$\varphi$ and $\overline{\varphi}$ with $[\chi_N,\varphi]_N = 1$.
Since $I_G(\varphi)$ splits over
$N$, $\varphi$ is extendible to $I_G(\varphi)$, and hence
$\chi_{\{\varphi,\overline{\varphi}\}}(1)=|G:I_G(\varphi)|$,
which is an even number since $\varphi$ is not $G$-invariant.

Applying the above argument to
$\lambda_i \in \Omega_i$, we get a strongly real character $\chi_i \in \Irr(G)$ of
even degree $ \geq k_i \geq 2$; in particular, $s \leq (\sum^s_{i=1}k_i)/2$ and so
$$3\sum^s_{i=1}k_i-4s \geq \sum^s_{i=1}k_i \geq n_{1,+}(G).$$
It follows that the average degree of these $s$ characters $\chi_1, \ldots ,\chi_s$ and the
(strongly) real linear characters of $G$ is at least
$$\frac{n_{1,+}(G) + \sum^s_{i=1}k_i}{n_{1,+}(G)+s} \geq \frac{4}{3}.$$
Since any other strongly real irreducible character of $G$ has
degree at least 2, we conclude that $\acd_{2,+}(G)\geq 4/3$.
\end{proof}

The proof of Theorem \ref{theorem-main-3} is completed by

\begin{theorem}\label{theorem-strongly-real-normal2Sylow}
Let $G$ be a finite group. If $\acd_{2,+}(G)< 4/3$ then $G$ has a
normal Sylow $2$-subgroup.
\end{theorem}

\begin{proof}
We assume that the statement is false, and let $G$ be a minimal
counterexample. In particular, $G$ is non-abelian. Let $N$ be a
minimal normal subgroup of $G$ such that $N\subseteq G'$. As before,
it follows that the degree of every strongly real, irreducible
character of $G$ whose kernel does not contain $N$ is at least 2,
and hence $\acd_{2,+}(G/N)\leq \acd_{2,+}(G)<4/3$. By the minimality
of $G$, we then deduce that $G/N$ has a normal Sylow $2$-subgroup,
and thus $NP \unlhd G$, where $P$ is a Sylow $2$-subgroup of $G$.

Since $\acd_{2,+}(G)< 4/3<2$,
Theorem~\ref{theorem-strongly-real-solvable} guarantees that $G$ is
solvable, and so $N$ is elementary abelian. If $N$ has even order,
then $NP=P$ and we would be done. So we may assume that $N$ is an
elementary abelian group of odd order, and moreover, $G$ has no
non-trivial normal $2$-subgroup.

We observe that each strongly real linear character of $G$ restricts
to a strongly real linear character of $NP$, and moreover, by
Lemma~\ref{lemma-strong-real}, each strongly real linear character
of $NP$ lies under a unique strongly real linear character of $G$.
It follows that $n_{1,+}(NP)=n_{1,+}(G)$. This and
Lemma~\ref{lemma-strong-real} imply that $\acd_{2,+}(NP)\leq
\acd_{2,+}(G)<4/3$. Thus, if $NP<G$, then by the minimality of $G$,
we have $P\lhd NP$ so that $NP=N\times P$, and hence $P\lhd G$, a
contradiction. So we may assume that $G=N\rtimes P$. Since $G$ has
no non-trivial normal $2$-subgroup, we see that $N$ is the unique
minimal normal subgroup of $G$.

Now we have all the hypotheses of
Proposition~\ref{proposition-strongly-real}, and therefore we
conclude that $\acd_{2,+}(G)\geq 4/3$. This contradiction completes
the proof.
\end{proof}

\section*{Acknowledgement} We are grateful to the referee
for helpful comments and suggestions that have significantly
improved the exposition of the paper.


\end{document}